\documentclass{article}[10pt]
\usepackage{amssymb}
\usepackage{amsmath, amsthm}
\usepackage{color}
\usepackage{graphicx}
\newtheorem{theorem}{Theorem}[section]
\newtheorem{lemma}{Lemma}[section]
\newtheorem{corollary}{Corollary}[section]
\newtheorem{definition}{Definition}[section]
\newtheorem{proposition}{Proposition}[section]
\newtheorem{remark}{Remark}[section]
\newtheorem{assumption}{Assumptions}[section]

\begin{document}
\title
{Convergent Iteration in Sobolev Space 
for Time Dependent Closed Quantum
Systems}
\author{Joseph W. Jerome\footnotemark[1]}
\date{}
\maketitle
\pagestyle{myheadings}
\markright{Convergent Iteration in Sobolev Space}
\medskip

\vfill
{\parindent 0cm}

{\bf 2010 AMS classification numbers:} 
35Q41; 47D08; 47H09; 47J25; 81Q05.  

\bigskip
{\bf Key words:} 
Time dependent quantum systems; 
time-ordered evolution operators; Newton iteration; quadratic convergence. 

\fnsymbol{footnote}
\footnotetext[1]{Department of Mathematics, Northwestern University,
                Evanston, IL 60208. \\In residence:
George Washington University,
Washington, D.C. 20052\\
{URL: {\tt {http://www.math.northwestern.edu/$\sim$jwj}}} $\;$ 
{Email: {\tt {jwj620@gwu.edu}}}
}

\begin{abstract}
Time dependent quantum systems have become indispensable in science and
its applications, particularly at the atomic and molecular levels.
Here, we discuss the approximation of closed 
time dependent quantum systems on bounded domains, via 
iterative methods in Sobolev space based upon evolution operators. 
Recently, existence and uniqueness of weak solutions were demonstrated
by a contractive fixed point mapping defined by the evolution operators. 
Convergent successive approximation is then guaranteed. 
This article uses the same mapping to define 
quadratically convergent Newton and approximate Newton methods.
Estimates for the constants used in the convergence estimates are provided.
The evolution operators are ideally suited to serve as the
framework for this operator approximation theory, 
since the Hamiltonian is time-dependent. In addition, the 
hypotheses required to guarantee 
quadratic convergence of the Newton iteration build naturally upon the
hypotheses used for the existence/uniqueness theory. 
\end{abstract}

\section{Introduction}
Time dependent density functional theory (TDDFT) dates from the seminal
article \cite{RG}. 
A current account of the subject 
and a development of the mathematical model may be found in \cite{U}. 
For the earlier density functional theory (DFT), 
we refer the reader to \cite{KS,KV}. 
Our focus in this article is TDDFT only.
Closed quantum systems on bounded domains of
${\mathbb R}^{3}$ were analyzed in \cite{JP,J1}, 
via time-ordered evolution operators. The article \cite{JP} includes 
simulations 
based on approximations of the evolution operator, employing a spectral 
algorithm. 
The article \cite{J1} extended the existence results of \cite{JP} 
to include weak
solutions via a strict contraction argument for an operator $K$;
\cite{J1} also includes the exchange-correlation component of the 
Hamiltonian potential not included in \cite{JP}.
TDDFT is a significant field for applications (see 
\cite{CMR,tddft,YB,LHP,MG}), including the expanding field of chemical
physics, which studies the response of atoms and molecules to external
stimuli. 
By permitting time dependent potentials,
TDDFT extends the nonlinear Schr\"{o}dinger equation, 
which has been studied extensively (\cite{CH,Caz}). 

In this article, we build upon \cite{J1}  by introducing a Newton iteration 
argument for $I -K$. This is examined at several levels, including
classical Newton iteration and also 
that of approximate Newton iteration, defined by residual
estimation.
It is advantageous that the existence/uniqueness theory of \cite{J1} 
employs strict contraction on a domain with an `a priori' norm bound. 
One consequence is that the local requirements of Newton's method can, in
principle, be satisfied by preliminary Picard iteration (successive
approximation).
The results of this article should be viewed as advancing understanding
beyond that of an existence/uniqueness theory; 
they are directed toward ultimately
identifying a successful constructive approach to obtaining solutions.
In the following subsections of the introduction, we familiarize the reader with
the model, and summarize the basic results of \cite{J1}, 
which serve as the starting point for the present article.
In this presentation, we provide explicit estimates 
for the domain and contraction constant used for the application of the
Banach contraction theorem. 
Section two cites and derives essential operator results regarding Newton
iteration in Banach spaces. 
Section three is devoted to an exact Newton iteration, with
quadratic convergence, for the quantum  TDDFT model, whereas section four
considers an approximate quadratically convergent Newton
iteration for the TDDFT model. 
Section five is a brief `Conclusion' section.
Appendices are included, which state the hypotheses used in \cite{J1} 
(Appendix A), 
basic definitions of the norms and function spaces adopted for this
article (Appendix B), and a general Banach space lemma
characterizing quadratic convergence for approximate Newton iteration
(Appendix C). 

The constants which appear in the analysis are directly 
related to the components of the potential. The external potential and the
Hartree potential present no problem.
However, as observed in \cite{U}, there is no explicit universally 
accepted representation of the exchange-correlation potential, which is
required to be non-local in our approach. 
It follows that the explicit convergence estimates we present are 
strongest in the absence of this potential, or 
in the case when concrete approximations are employed. 
\subsection{The model}
TDDFT includes an external potential, the Hartree potential, 
and the exchange-correlation potential. 
If $\hat H$ denotes
the Hamiltonian operator of the system, then the state $\Psi(t)$ of the
system obeys the nonlinear Schr\"{o}dinger equation,
\begin{equation}
\label{eeq}
i \hbar \frac{\partial \Psi(t)}{\partial t} = \hat H \Psi(t).
\end{equation}
Here, $\Psi = \{\psi_{1}, \dots, \psi_{N}\}$ and the charge
density $\rho$ is defined by 
$$ \rho({\bf x}, t) = |\Psi({\bf x}, t)|^{2} = 
\sum_{k = 1}^{N} |\psi_{k} ({\bf x}, t)|^{2}.
$$ 
For well-posedness, an initial condition,
\begin{equation}
\label{ic}
\Psi(0) = \Psi_{0}, 
\end{equation}
consisting of $N$ orbitals,
and boundary conditions must be adjoined. 
We will assume that the particles are confined to a
bounded region $\Omega \subset {\mathbb R}^{3}$ and that
homogeneous Dirichlet boundary conditions hold for the evolving quantum
state within a closed system. In general, $\Psi$ denotes a finite vector
function of space and time. 
\subsection{Specification of the Hamiltonian operator}
We study effective potentials $V_{\rm e}$ 
which are real scalar functions of the form,
$$
V_{\rm e} ({\bf x},t, \rho) = V({\bf x}, t) + 
W \ast \rho + \Phi({\bf x}, t, \rho).
$$
Here, 
$W({\bf x}) = 1/|{\bf x}|$ 
and the convolution
$W \ast \rho$ 
denotes the Hartree potential. If $\rho$ is extended as zero outside
$\Omega$, then, for ${\bf x} \in \Omega$, 
$$
W \ast \rho \; ({\bf x})=\int_{{\mathbb R}^{3}} 
W({\bf x} -{\bf y}) \rho({\bf y})\;d {\bf y},
$$
which depends only upon values $W({\xi})$,
$\|{\xi}\|\leq 
\mbox{diam}(\Omega)$. We may redefine $W$ 
smoothly outside this set,
so as to obtain a function of compact support for which Young's inequality
applies. 
$\Phi$ represents a time history of $\rho$:
$$
\Phi({\bf x}, t, \rho)= \Phi_{0}({\bf x}, 0, \rho) + 
\int_{0}^{t} \phi({\bf x}, s, \rho) \; ds.
$$  
As explained in \cite[Sec.\thinspace 6.5]{U}, $\Phi_{0}$ 
is determined by
the initial state of the Kohn-Sham system and the initial state of the
interacting reference system with the same density and divergence of the
charge-current.
The Hamiltonian operator is given by, for effective
mass $m$ and normalized Planck's constant $\hbar$,
\begin{equation}
\hat H  
= -\frac{\hbar^{2}}{2m} \nabla^{2} 
 +V({\bf x}, t) + 
W \ast \rho + \Phi({\bf x}, t, \rho).
\label{Hamiltonian}
\end{equation}
\subsection{Definition of weak solution}
The solution $\Psi$ is continuous from the time interval $J$, to be
defined shortly,  
into the finite energy Sobolev space
of complex-valued 
vector functions which vanish in a generalized sense on the boundary, 
denoted $H^{1}_{0}(\Omega)$. One writes: $\Psi \in C(J; H^{1}_{0})$. 
The time derivative is continuous from $J$ 
into the dual $H^{-1}$ of $H^{1}_{0}$. 
One writes: $\Psi \in C^{1}(J; H^{-1})$.  
The spatially dependent test functions $\zeta$ are
arbitrary in $H^{1}_{0}$. 
The duality bracket is denoted $\langle f, \zeta \rangle$. 
\begin{definition}
For $J=[0,T]$,  
the vector-valued function $\Psi = \Psi({\bf x}, t)$ is a  
weak solution of (\ref{eeq}, \ref{ic}, \ref{Hamiltonian}) if 
$\Psi \in C(J; H^{1}_{0}(\Omega)) \cap C^{1}(J;
H^{-1}(\Omega)),$ if $\Psi$ satisfies the initial condition 
(\ref{ic}) for $\Psi_{0} \in H^{1}_{0}$, 
and if $\forall \; 0 < t \leq T$: 
\vspace{.25in}
\begin{equation}
i \hbar\langle \frac{\partial\Psi(t)}{\partial t},
\zeta \rangle  = 
\int_{\Omega} \frac{{\hbar}^{2}}{2m}
\nabla \Psi({\bf x}, t)\cdotp \nabla { \zeta}({\bf x}) 
+ V_{\rm e}({\bf x},t,\rho) \Psi({\bf x},t) { \zeta}({\bf x})
d{\bf x}. 
\label{wsol}
\end{equation}
\end{definition}
\subsection{The associated linear problem}
The approach to solve the nonlinear problem (\ref{wsol}) is to define a
fixed point mapping $K$. For each $\Psi^{\ast}$ in the domain
$C(J;H^{1}_{0})$ of $K$ 
we produce the image $K \Psi^{\ast} = \Psi$ by the following
steps.  
\begin{enumerate}
\item
$\Psi^{\ast} \mapsto \rho =
 \rho({\bf x}, t) = |\Psi^{\ast}({\bf x}, t)|^{2}$.
\item
$\rho \mapsto \Psi$ by the solution of the 
{\it associated} linear problem (\ref{wsol})
where the potential $V_{\rm e}$ uses $\rho$ 
in its final argument.
\end{enumerate}
In general, $\Psi \not= \Psi^{\ast}$ unless $\Psi$ is a fixed
point of $K$. 

In order to construct a fixed point, one introduces the linear evolution
operator $U(t,s)$: for given $\Psi^{\ast}$ in
$C(J;H^{1}_{0})$, 
set $U(t,s) = U^{\rho}(t,s)$ so that
\begin{equation}
\label{evolution} 
\Psi(t) = U^{\rho}(t,0) \Psi_{0}. 
\end{equation}
For each $t$, we interpret $\Psi(t)$ 
as a function (of ${\bf x}$).
Moreover, the effect of the evolution operator is to obtain 
$\Psi = K \Psi^{\ast}$ since the operator is generated
from (\ref{wsol}) with specified $\rho$.
\subsection{Discussion of the evolution operator}
The evolution operator used here and in \cite{J1}
was introduced in two fundamental
articles \cite{K1,K2} 
by Kato in the 1970s. A description of Kato's theory can be
found in \cite{J2}. For the application to (\ref{wsol}), one identifies 
the frame space
with the dual space $H^{-1}$ and the smooth space with the finite energy space
$H^{1}_{0}$. A significant step is to show that the operators 
$(-i/{\hbar}) \hat H(\rho)$ generate contraction semigroups on the frame space,
which remain stable on the smooth space. If one can demonstrate these 
properties, then the evolution operator exists and can be used as in
(\ref{evolution}) to retrieve the solution of the initial value problem. 
\subsection{Discussion of the result}
The following theorem was proved in \cite{J1}. We include a short appendix 
which describes the hypotheses under which Theorem \ref{EU} holds. 
\begin{theorem}
\label{EU}
There is a closed ball ${\overline {B(0, r)}} \subset C(J; H^{1}_{0})$ 
on which $K$ is invariant. 
For $t$ sufficiently small, $K$ defines a strict contraction. The contraction
constant is independent of the restricted time interval, so that the unique
fixed point can be continued globally in time. 
In particular, for any interval $[0,T]$, the system (\ref{wsol}) 
has a unique solution 
which coincides with this fixed point.  
\end{theorem}
The proof uses the 
Banach contraction mapping theorem. 
We will elaborate on this after the following subsection
because it is required for later estimates.
\subsection{A variant of conservation of energy}
\label{cofe}
If the functional ${\mathcal E}(t)$ is defined 
for $0 < t \leq T$ by,  
$$
\int_{\Omega}\left[\frac{{\hbar}^{2}}{4m}|\nabla \Psi|^{2} 
+ 
\left(\frac{1}{4}(W \ast |\Psi|^{2})+ \frac{1}{2} 
(V + \Phi)\right) |\Psi|^{2}\right]d{\bf x},
$$
then the following identity holds for the unique solution:
\begin{equation}
{\mathcal E}(t)={\mathcal E}(0)
+
\frac{1}{2}\int_{0}^{t}\int_{\Omega}[(\partial V/\partial s)({\bf x},s)
+ \phi({\bf x}, s, \rho)]
|\Psi|^{2}\;d{\bf x}ds,
\label{consener}
\end{equation}
where 
${\mathcal E}(0)$
is given by
$$ 
\int_{\Omega}\left[\frac{{\hbar}^{2}}{4m}|\nabla
\Psi_{0}|^{2}+\left(\frac{1}{4} 
(W\ast|\Psi_{0}|^{2})+\frac{1}{2}
(V  + \Phi_{0}) \right)|\Psi_{0}|^{2}\right]
\;d{\bf x}.
$$
The functional ${\mathcal E}(t)$ is related to the physical
energy $E(t)$ of the system, defined by
$
E(t) = \langle{\hat H}(t) \Psi(t), {\bar \Psi(t)} \rangle: 
$
$$
{\mathcal E}(t) = \frac{1}{2} \left(E(t) - \frac{1}{2} \langle \Psi(t),
(W \ast |\Psi(t)|^{2}) {\bar \Psi(t)} \rangle \right).
$$
\subsection{Restriction of the domain of $K$}
We discuss the invariant closed ball cited in Theorem \ref{EU}. 
Consider the evolution operator $U^{v}$ corresponding to 
$v = 0$ (zero charge),
and use this as reference:
$$
U^{\rho}(t,0)\Psi_{0} = 
[U^{\rho}(t,0)\Psi_{0} - 
U^{v}(t,0)\Psi_{0}] + 
U^{v}(t,0)\Psi_{0}. 
$$
This is valid for the first stage of the continuation process indicated in
Theorem \ref{EU}. For subsequent stages, $\Psi_{0}$ is replaced by the
solution evaluated at discrete points $t_{k}$. 
We define $r$ as follows.
\begin{equation}
\label{definitionr}
r = 2 \|U\|_{\infty, H^{1}_{0}} \max (\|\Psi_{0}\|_{H^{1}_{0}}, r_{0}).
\end{equation}
Here, $r_{0}$ is a bound for the $C(J; H^{1}_{0})$ norm of the
solution, derived from the preceding subsection, and discussed in the
following corollary. By use of identity 
(\ref{IDENTITY}) below, the difference term in the
above representation can be controlled by the size of $t$, 
not to exceed $r/2$ (see (\ref{IDENTITY})). In particular,
the closed ball is invariant. 

\begin{corollary}
The number $r^{2}_{0}$ can be chosen as any upper bound for 
\begin{equation}
\label{firstr}
\frac{4m}{\hbar^{2}} \sup_{0\leq t \leq T}{\cal E}(t) +
\|\Psi_{0}\|^{2}_{L_{2}}.
\end{equation}
In particular, if we denote by ${\cal V}$ the sum $V + \Phi$, then  
we may make the choice,
\begin{equation}
\label{secondr}
r^{2}_{0} = 
\frac{4m}{\hbar^{2}}\left[{\cal E}_{0} +  \frac{T}{2}\sup_{x\in \Omega,
t\leq T} \left|\partial {\cal V}/\partial t \right|
\|\Psi_{0}\|^{2}_{L_{2}}\right]
+ \|\Psi_{0}\|^{2}_{L_{2}}.
\end{equation}
\end{corollary}
\begin{proof}
The upper bound contained in (\ref{firstr}) uses definitions and
nonnegativity. The choice in (\ref{secondr}) uses the identity 
(\ref{consener}). Both use $L^{2}$ norm invariance for $\Psi$.
\end{proof}

\subsection{The contractive property: Role of the evolution operator}
\label{CP}
There is an integral operator identity satisfied by the evolution operator
which permits the estimation of the metric distance between 
$K \Psi^{\ast}_{1}$ and 
$K \Psi^{\ast}_{2}$. Note that separate evolution operators 
$U^{\rho_{1}}(t,s)$
and $U^{\rho_{2}}(t,s)$ are generated, for $\rho_{1} = |\Psi_{1}^{\ast}|^{2}, 
\rho_{2} = |\Psi_{2}^{\ast}|^{2}$. 
One has the following identity (see \cite{J2}):
\begin{equation}
U^{\rho_{1}}\Psi_{0}(t) - U^{\rho_{2}}\Psi_{0}(t) =  
\frac{i}{\hbar}\int_{0}^{t} U^{\rho_{1}}(t,s)[{\hat H}(s, \rho_{1}) -
{\hat H}(s, \rho_{2})]U^{\rho_{2}}(s,0)\Psi_{0} \; ds.
\label{IDENTITY}
\end{equation}
The $H^{1}_{0}$-norm of the evolution operators can be 
uniformly bounded in $t,s$ by a 
constant (see below).
The following Lipschitz condition, 
on the (restricted) domain 
${\overline {B(0, r)}} \subset C(J; H^{1}_{0})$ of $K$ 
was used in \cite{J1}:
\begin{equation}
\label{LipforV}
\|[V_{\rm e}({\rho_{1}})-
V_{\rm e}({\rho_{2}})]\psi\|_{C(J;H^{1}_{0})} \leq  
C\|\Psi^{\ast}_{1} -
\Psi^{\ast}_{2} 
\|_{C(J;H^{1}_{0})} \|\psi\|_{H^{1}_{0}}. 
\end{equation}
Here $C$ is a fixed positive constant (see below) 
and $\psi$ is arbitrary in $H^{1}_{0}$. 
Inequality  (\ref{LipforV}) (see 
Theorem \ref{hartreeLip} and (\ref{ecfollowsH}) to follow)
can be used to verify contraction, for sufficiently small $t$,
in terms of the distance between
$\Psi^{\ast}_{1}$ and 
$\Psi^{\ast}_{2}$. In other words, one begins by replacing
$J$ by $[0, t]$ so that $\gamma < 1$. 
Continuation to $t = T$ occurs 
in a finite number of steps.  
We now present a lemma which estimates the Lipschitz
(contraction) constant $\gamma$
of $K$ on ${\overline {B(0,r)}}$.
In the lemma, and elsewhere, we use the notation,
\begin{equation}
\label{uniformstU}
\|U\|_{\infty,H^{1}_{0}} :=
\sup_{t \in J, s \in J} \|U(t,s)\|_{H^{1}_{0}}.
\end{equation}
\begin{lemma}
\label{contractionconstant}
The mapping $K$ is Lipschitz continuous on 
${\overline {B(0,r)}}$, with Lipschitz constant $\gamma = \gamma_{t}$ 
estimated by
\begin{equation}
\label{gamma}
 \gamma = \gamma_{t} 
\leq \frac{Ct}{\hbar} \|U\|_{\infty, H^{1}_{0}}^{2}\;\|\Psi_{0}\|_{H^{1}_{0}}. 
\end{equation}
Here, $C$ is the constant of (\ref{LipforV}). $C$ can be estimated
precisely in the case when $V_{\rm e}(\rho)$ is independent of $\Phi$.
If $E_{1}$ is the Sobolev embedding constant associated
with the embedding of $H^{1}_{0}$ into $L^{6}$, then $C = 2r C_{0}$, where
\begin{equation}
\label{constantC}
C_{0} = E_{1}^{2}[ E_{1}\|\nabla W\|_{L^{!}} + |\Omega|^{2/3}\|W\|_{L^{2}} 
+ E_{1}^{2}|\Omega|].  
\end{equation}
Here, $W$ is the Hartree convolution kernel, and $r$ has the meaning of
(\ref{definitionr}). Also, in this case, we have the estimate,
\begin{equation}
\label{estEVY}
 \|U\|_{\infty, H^{1}_{0}} \leq \exp \left(\frac{E_{1}\hbar^{2}T}{2 m}
\left[2r E_{1}^{2}  
 \|\nabla W\|_{L^{1}} + \|\nabla V\|_{C(J; L^{3})}\right]\right).
\end{equation}
If $\Phi$ is included, the estimate for $C_{0}$ in (\ref{constantC}) is
incremented by the constant appearing in (\ref{ecfollowsH}). 
Also, the rhs of (\ref{estEVY}) is modified: within [ $\cdot$], 
one adds a uniform estimate for $\|\nabla \Phi \|_{L^{3}}$.
\end{lemma}
\begin{proof}
The estimate (\ref{gamma}) is a direct consequence of estimating the
$H^{1}_{0}$ norm of (\ref{IDENTITY}). 
The estimation of the constant $C$ in (\ref{gamma}) follows from tracking
the constants appearing in 
the proofs of Lemma \ref{3.1} and Theorem \ref{hartreeLip} to
follow. The uniform bound for 
$ \|U\|_{\infty, H^{1}_{0}}$ is established in the construction of Kato;
we used the explicit estimate given in \cite[Corollary 6.3.6, p. 229]{J2},
together with the fact that the evolution operators are contractive on the
dual space, and the fact that we employ the canonical  isomorphism between
$H^{1}_{0}$ and $H^{-1}$.
\end{proof}

\section{Background Results for Newton Iteration in Banach Space}
We have seen previously that, 
by appropriate use of successive approximation, we can, in principle,  
determine approximations of arbitrary prescribed accuracy.
In fact, if it is required of the $n$th iterate $\psi_{n}$ 
of a contractive map with contraction constant $q$ and fixed point
$\psi$ that 
\begin{equation}
\label{SA}
\|\psi_{n} - \psi \| \leq \epsilon,
\end{equation}
then the well-known successive approximation estimate,
$$
\|\psi_{n} - \psi \| \leq 
\frac{q^{n}}{1 - q} 
\|\psi_{1} - \psi_{0} \| \leq \epsilon,
$$
gives the value of $n$, which must be satisfied. 
It is natural then to investigate rapidly converging local methods, 
specifically, Newton's method.
In this article, we will discuss
both exact and approximate quadratically convergent Newton methods. The latter  
permit approximate derivative inverses, and thus approximate
Newton iterations.
\subsection{The core convergence result} 
We have included in Appendix \ref{appendixC} a 
core Lemma 
which permits both exact
and approximate Newton methods in Banach spaces. It is based upon earlier
work of the author \cite{J3} and will serve as a resource result.
In the estimates, $h\leq 1/2$ should be viewed as a discretionary
parameter, and $\kappa, \sigma$ as constraining parameters to be
determined. The
locality of Newton's method is incorporated in the requirement that the
initial residual not exceed $\sigma^{-1}$.
The parameter $\alpha < 1$ positions $u_{0}$ in the interior of an
appropriate closed ball. The parameter $\tau < 1$ is part of the defining
equation for invertibility of the derivative map.
\subsection{An exact Newton method for fixed point maps} 
We discuss the case where the exact inverse is employed for the 
Fr\'{e}chet derivative. We derive a general operator result, 
applied in section three. 
\begin{proposition}
\label{VK}
Let ${\mathcal O}$ be an open subset of a Banach space $X$, 
and let $P: {\mathcal O} \mapsto X$ be such that $P$ is Lipschitz continuously 
Fr\'{e}chet differentiable on ${\mathcal O}$, 
and such that 
$S^{\prime}(x_{0}) = I - P^{\prime}(x_{0})$
is invertible for some $x_{0} \in {\mathcal O}$. 
Then there is a suitable closed ball 
$B_{\delta} = {\overline {B(x_{0}, \delta)}}$ 
for which
$S^{\prime}(v)$ is invertible for each $v \in B_{\delta}$. 
If $0 < \alpha < 1$ and $h \leq 1/2$, then there are choices of $\kappa$ and 
$\sigma$ 
such that, if $u_{0} \in B_{\alpha \delta}$ satisfies 
the consistency condition,
\begin{equation}
\label{smallres}
\|Su_{0}\| \leq \sigma^{-1},
\end{equation}
then the hypotheses (\ref{Kanone}, \ref{Kantwo}) 
hold with $G_{v} = [S^{\prime}(v)]^{-1}$.
If $P$ is a strict contraction on $B_{\delta}$, with unique fixed point
$x_{0}$, then a starting iterate $u_{0} \in B_{\alpha \delta}$ can be
found for which (\ref{smallres}) holds.
In particular, (\ref{R-quadratic}) holds for the Newton iteration with 
$u$ identified with $x_{0}$.
\end{proposition}
\begin{proof}
According to the Lipschitz property satisfied 
by $P^{\prime}$ on ${\mathcal O}$, 
say Lipschitz constant $c$, 
for any specified positive $\tau < 1$,
we can find $\delta > 0$ such that 
$$
\|[I - P^{\prime}(x_{0})]^{-1} \| \; 
\|P^{\prime}(x_{0}) - P^{\prime}(v)\|  
\leq \tau < 1, \; \mbox{if} \; 
\|x_{0} - v\| \leq  \delta.
$$
By a standard perturbation lemma \cite{G}, 
it follows that $I - P^{\prime}(v)$ is invertible in the closed
ball of radius $\delta$, centered at $x_{0}$. 
The perturbation lemma gives the uniform bound
for the norms of the inverses $[S^{\prime}(v)]^{-1}$:
\begin{equation}
\label{definitionkappa}
\kappa:= \frac{\|[S^{\prime}(x_{0})]^{-1}\|}{1 - \tau}. 
\end{equation}
This gives (\ref{Kanone}).
Now choose $\sigma$ sufficiently large so that the following two 
inequalities hold:
$$
c \kappa^{2} \leq h \sigma, \; \frac{\kappa}{h \sigma} \leq (1-
\alpha)\delta.
$$
Suppose $u_{0} \in B_{\alpha \delta}$.
To obtain (\ref{Kantwo}), 
we employ a version of the fundamental theorem of 
calculus for Fr\'{e}chet derivatives (valid in Fr\'{e}chet spaces \cite{H}):
\begin{equation*}
S(u_{k}) = 
\int_{0}^{1} [S^{\prime}(u_{k-1} +t (u_{k} - u_{k-1})) - S^{\prime}(u_{k-1})]
(u_{k} - u_{k-1}) \; dt.  
\end{equation*}
By estimating this integral, we obtain via (\ref{Kanone}):
$$
\|S(u_{k})\| \leq \frac{c}{2} \|u_{k} - u_{k-1}\|^{2}  
\leq \frac{c \kappa^{2}}{2}\|S(u_{k-1})\|^{2}. 
$$
By the choice of $\sigma$, we thus obtain (\ref{Kantwo}). 
We now consider the residual condition for $u_{0}$.
Suppose $P$ is a strict contraction, with contraction constant $q$ and
fixed point $x_{0}$. 
In order to obtain the consistency and convergence of the iterates, we
select $\epsilon = \sigma^{-1}/(1 + q)$ in (\ref{SA}) and identify $u_{0}$ 
with the 
$nth$ successive approximation $p_{n} = P p_{n-1}$ defined by $P$.
This works since, by interpreting (\ref{SA}), we have
$$
\|Sp_{n}\| = \|Pp_{n} - p_{n}\| = \|P(p_{n} - x_{0}) + (x_{0} - p_{n})\|  
\leq (q + 1) \epsilon \leq \sigma^{-1}.
$$
The Lemma of Appendix C now applies.
\end{proof}
\begin{corollary}
\label{exactcombined}
Suppose that $P$ is a strict contraction, with contraction constant $q$.
Suppose the hypotheses of Proposition \ref{VK} are satisfied and suppose
$c$ is the Lipschitz constant of $P^{\prime}$. 
If $S = I - P$, define
$\kappa$ by (\ref{definitionkappa}) and $\sigma$ by
\begin{equation}
\label{explicitsigma}
\sigma = \max \left(\frac{c \kappa^{2}}{h}, 
\frac{c \kappa^{2} (1 - \tau)}{h (1- \alpha) \tau} \right).  
\end{equation} 
If $u_{0}$ is defined by successive approximation to be the iterate
$p_{n}$ satisfying (\ref{SA}) with $\epsilon = \frac{1}{\sigma (1 + q)}$,
then the exact Newton iteration is quadratically convergent 
as in (\ref{R-quadratic}).
\end{corollary}
\begin{remark}
Note that, along the curve $1 - \tau = \tau (1 - \alpha)$, in the open square 
$1/2 < \tau < 1, 0 < \alpha < 1$,  
both expressions on the rhs of (\ref{explicitsigma}) are equal. 
\end{remark}
\section{Classical Newton Iteration for the Quantum System}
This section is devoted to the exact Newton method for our
model, as based upon the fixed point mapping $K$, which is identified with
$P$ in Proposition \ref{VK}.  
\subsection{Fundamental inequality for the Hartree potential}
The results of this subsection were used in the estimation of the
contraction constant $\gamma$ of $K$ (see Lemma
\ref{contractionconstant}).
We begin with a lemma for the convolution of $W = 1/|x|$ with products of
$H^{1}_{0}(\Omega)$ functions. This is later applied to the Hartree
potential. 
\begin{lemma}
\label{3.1}
Suppose that $f,g$ and $\psi$ are arbitrary functions in $H^{1}_{0}$, and
set $w = W \ast (fg)$. Then
\begin{equation}
\label{conprod}
\|w \psi\|_{H^{1}_{0}} \leq C \|f\|_{H^{1}_{0}} \|g\|_{H^{1}_{0}}
\|\psi\|_{H^{1}_{0}}
\end{equation}
where $C$ is a generic constant.
\end{lemma}
\begin{proof}
We claim that
the following two inequalities are sufficient to prove 
(\ref{conprod}).
\begin{eqnarray*}
\|w\|_{W^{1,3}} &\leq& C_{1} \|f\|_{H^{1}_{0}}\|g\|_{H^{1}_{0}}  \\
\|w\|_{L^{\infty}} &\leq& C_{2} \|f\|_{H^{1}_{0}}\|g\|_{H^{1}_{0}}. 
\end{eqnarray*}
Suppose that these two inequalities hold. We show that the lemma follows. 
We use duality \cite[Theorem 4.3, p. 89]{Rudin} to estimate the
$H^{1}_{0}$ norm of $w \psi$. To begin, we have
$$
\|w \psi \|_{H^{1}_{0}} = 
\sup_{\|\omega\|_{H^{1}_{0}} \leq 1}\left\{\left|\int_{\Omega}\nabla
(w\psi) \cdotp \nabla \omega + 
\int_{\Omega}
w \psi \omega \right| \right\}. 
$$
Furthermore,
$$
\nabla (w \psi) = (\nabla w) \psi + w \nabla \psi,
$$
so that
$$
\nabla (w \psi) \cdotp \nabla \omega = 
(\nabla w) \psi \cdotp \nabla \omega + w \nabla \psi \cdotp
\nabla \omega.
$$
Upon taking the $L^{1}$ norm of both sides of this relation, we obtain,
$$
\|\nabla (w \psi) \cdotp \nabla \omega \|_{L^{1}} \leq 
\|(\nabla w) \psi \cdotp \nabla \omega \|_{L^{1}} + 
\|w \nabla \psi \cdotp
\nabla \omega\|_{L^{1}}.
$$
We now estimate each of these terms via the H\"{o}lder and Sobolev
inequalities. For the first term,
\begin{equation}
\label{I}
\|(\nabla w) \psi \cdotp \nabla \omega \|_{L^{1}} \leq
\|\psi\|_{L^{6}} \|\nabla w \|_{L^{3}} \|\nabla \omega \|_{L^{2}}
\leq E_{1} \|\psi \|_{H^{1}_{0}} C_{1} \|f \|_{H^{1}_{0}} \|g \|_{H^{1}_{0}}, 
\end{equation}
where $E_{1}$ is a Sobolev embedding constant 
($H^{1}_{0} \hookrightarrow  L^{6}$), 
and where we have used the
first inequality of the claim, together with the norm assumption on 
$\omega$.
For the second term above,
\begin{equation}
\label{II}
\|w \nabla \psi \cdotp
\nabla \omega\|_{L^{1}} \leq \|w\|_{L^{\infty}} \|\nabla \psi
\|_{L^{2}} \|\nabla \omega\|_{L^{2}}
\leq C_{2} \|f \|_{H^{1}_{0}} \|g\|_{H^{1}_{0}} 
\|\psi \|_{H^{1}_{0}}, 
\end{equation}
where we have used the
second inequality of the claim, together with the norm assumption on 
$\omega$.
We estimate the final term in the supremum. 
\begin{equation}
\label{III}
\|w \psi 
\omega\|_{L^{1}} \leq \|w\|_{L^{3}} \|\psi\|_{L^{3}}
\|\omega\|_{L^{3}}
\leq E_{2}^{2} 
\|\psi \|_{H^{1}_{0}} C_{1} \|f\|_{H^{1}_{0}} 
\|g\|_{H^{1}_{0}}, 
\end{equation}
where $E_{2}$ is a Sobolev embedding constant 
($H^{1}_{0} \hookrightarrow  L^{3}$), 
and we have used the norm
assumption on $\omega$. 
By assembling inequalities (\ref{I}, \ref{II}, \ref{III}), we obtain the
estimate of the lemma if the claim is valid. 

We now verify each of the inequalities of the claim.
For the first, we have
\begin{eqnarray}
\nabla w = \nabla
[W \ast (fg)] 
&=& \nabla W \ast (fg) \nonumber \\ 
\|\nabla w\|_{L^{3}} &\leq& \|\nabla W\|_{L^{1}} 
\|fg\|_{L^{3}} \nonumber \\
&\leq& 
\|\nabla W\|_{L^{1}} 
\|f\|_{L^{6}}\; \|g\|_{L^{6}} \nonumber \\
&\leq& 
\|\nabla W\|_{L^{1}} 
E_{1}^{2} 
\|f\|_{H^{1}_{0}}\; \|g\|_{H^{1}_{0}}, 
\label{CI}
\end{eqnarray}
where we have used the Young, H\"{o}lder, and Sobolev inequalities.
For the second inequality of the claim, we have, by similar inequalities, 
\begin{eqnarray}
\|w \|_{\infty} &=& \sup_{x \in \Omega}\left| \int_{\Omega}
\frac{f(y)g(y) dy}{|x - y|}\right| \nonumber \\
&\leq& \|W\|_{L^{2}} \|fg\|_{L^{2}} \nonumber \\
&\leq& 
\|W\|_{L^{2}} 
\|f\|_{L^{4}}\; \|g\|_{L^{4}} \nonumber \\
&\leq& 
\|W\|_{L^{2}} 
E_{3}^{2} 
\|f\|_{H^{1}_{0}}\; \|g\|_{H^{1}_{0}}. 
\label{CII}
\end{eqnarray}
Here, $E_{3}$ is a Sobolev embedding constant
($H^{1}_{0} \hookrightarrow  L^{4}$). 
This establishes the claim, with specific estimates for $C_{1}, C_{2}$, 
 and the proof is concluded.
\end{proof}
\begin{remark}
Since convolution with $W/(4 \pi)$ provides a right inverse for the (negative)
Laplacian, we could infer the $L^{\infty}$ property of $w$ from the theory
of elliptic equations (see \cite[Th. 8.16, p.181]{GT}). However, we
require the explicit inequality (\ref{CII}).
\end{remark}
\begin{theorem}
\label{hartreeLip}
For the Hartree potential,
$$
H = W \ast \rho, \; \rho = |\Psi|^{2} \; (\Psi \in H^{1}_{0}),
$$
and $\psi \in H^{1}_{0}$,
we have,
$$
\|[W \ast (\rho_{1} - \rho_{2})]\psi \|_{H^{1}_{0}} \leq C 
\|\Psi_{1} - \Psi_{2}\|_{H^{1}_{0}} \|\psi\|_{H^{1}_{0}}.
$$
This inequality remains valid when $\Psi_{1}, \Psi_{2}$ are
functions in $C(J; H^{1}_{0})$. The appropriate norm subscripts are
replaced by 
$C(J; H^{1}_{0})$. In this case,  
$C$ is explicitly discussed in 
Lemma \ref{contractionconstant}. 
\end{theorem}
\begin{proof}
We apply the previous lemma after the simple factorization, 
$$
\rho_{1} - \rho_{2} = 
(|\Psi_{1}| - |\Psi_{2}|)
(|\Psi_{1}| + |\Psi_{2}|).
$$
In the lemma, we select
$$
f = |\Psi_{1}| - |\Psi_{2}|, \; 
g = |\Psi_{1}| + |\Psi_{2}|.
$$
The use of the reverse triangle inequality applied to $\|f\|_{H^{1}_{0}}$ 
and the standard triangle inequality applied to 
$\|g\|_{H^{1}_{0}}$ implies the estimate.
\end{proof}
\subsection{Hypothesis
for the exchange-correlation potential}
The hypotheses required of the exchange-correlation potential $\Phi$ in 
\cite{J1} are listed in the appendix. An additional hypothesis is required
if Fr\'{e}chet derivatives are required, as is the case in this article.
The hypothesis mirrors (\ref{conprod}). Note that $\Phi$ is defined in
section 1.2. The integrand $\phi$ may be a functional of $\rho$, as is the
case for the Hartree potential. 
\begin{assumption}
\label{3.1def}
In addition to the hypotheses on $\Phi$ expressed in the appendix, we
assume in addition the following. 
\begin{itemize}
\item
The functional derivative of $\Phi$ with respect to $\rho$ exists, is
defined on the product of 
$C(J; H^{1}_{0})$ functions, is linear, both in 
the product, and in the members of the product, 
and satisfies, for $z = (\partial \Phi/\partial
\rho)(fg)$, and $f,g \in 
C(J; H^{1}_{0}), 
\psi \in H^{1}_{0}$, 
\begin{equation}
\label{conprod2}
\|z \psi\|_{C(J; H^{1}_{0})} 
\leq C \|f\|_{C(J; H^{1}_{0})} \|g\|_{C(J; H^{1}_{0})}
\|\psi\|_{H^{1}_{0}},
\end{equation}
for some constant $C$.
\end{itemize}
\end{assumption} 
\subsection{G\^{a}teaux differentiability}
For application in this section, we rewrite 
(\ref{IDENTITY}) and 
(\ref{LipforV})  
in slightly simplified notation for use in this section:
\begin{equation}
U^{\rho_{1}}\Psi_{0}(t) - U^{\rho_{2}}\Psi_{0}(t) =  
\frac{i}{\hbar}\int_{0}^{t} U^{\rho_{1}}(t,s)[V_{\rm e}(s, \rho_{1}) -
V_{\rm e}(s, \rho_{2})]U^{\rho_{2}}(s,0)\Psi_{0} \; ds.
\label{IDENTITY2}
\end{equation}
\begin{equation}
\label{LipV}
\|[V_{\rm e}({\rho_{1}})-
V_{\rm e}({\rho_{2}})]\psi\|_{C(J;H^{1}_{0})} \leq  
C\|\Psi_{1} -
\Psi_{2} 
\|_{C(J;H^{1}_{0})} \|\psi\|_{H^{1}_{0}}. 
\end{equation} 
\begin{remark}
The constant $C$ in inequality (\ref{LipV}) 
and a uniform bound for the operator norm 
$\|U^{\rho}(t,s)\|_{\infty, H^{1}_{0}}$ 
are discussed in
Lemma \ref{contractionconstant}.  
We will represent such a bound by $\sup \|U^{\rho}\|$ below.
\end{remark}
\begin{lemma}
\label{uniform}
Suppose that $\Psi_{\epsilon}$ converges to $\Psi$ in $C(J; H^{1}_{0})$
as $\epsilon \rightarrow 0$. Then $U^{\rho_{\epsilon}}$ converges to
$U^{\rho}$ in the operator topology, uniformly in $t,s$.
In fact, the convergence is of order 
$O(\|\Psi_{\epsilon} - \Psi\|_{C(J;H^{1}_{0})})$.
\end{lemma}
\begin{proof}
We use the following operator representation for
$U^{\rho_{\epsilon}}(t,s) - U^{\rho}(t,s)$, for $0\leq s < t$, which is the
appropriate substitute for (\ref{IDENTITY2}): 
\begin{equation}
U^{\rho_{\epsilon}}(t,s) - U^{\rho}(t,s) =  
\frac{i}{\hbar}\int_{s}^{t} U^{\rho_{\epsilon}}(t,r)
[V_{\rm e}(r, \rho_{\epsilon}) -
V_{\rm e}(r, \rho)]U^{\rho}(r,s) \; dr.
\label{IDENTITY3}
\end{equation}
We estimate this operator at an arbitrary $\psi \in  
H^{1}_{0}$ of norm not exceeding one:
\begin{equation}
\label{upperbd}
(T/\hbar)\; 
\sup \|U^{\rho}\|
\sup_{0 \leq s \leq r \leq T} \|
[V_{\rm e}(r,\rho_{\epsilon})-V_{\rm e}(r,\rho)]U^{\rho}(r,s)\psi\|_{H^{1}_{0}}.
\end{equation}
An application of (\ref{LipV})  yields the
bound of a constant times
$$
\|\Psi_{\epsilon} - \Psi\|_{C(J;H^{1}_{0})},
$$
which completes the proof.
\end{proof}
\begin{proposition}
\label{Dif}
The operator $K$ is G\^{a}teaux differentiable on 
$C(J;H^{1}_{0})$. The derivative is given by (\ref{defKprime}) below, and
is a bounded linear operator on 
$C(J;H^{1}_{0})$, 
with bound given by (\ref{adapest}) below.
\end{proposition}
\begin{proof}
Let $\Psi$ be a given element of $C(J; H^{1}_{0})$ and set $\rho = |\Psi|^{2}$.
We begin with the formula (\ref{IDENTITY2}) 
and make the 
identification, 
$$\rho_{\epsilon} =  
|\Psi + \epsilon \omega|^{2}, \;  \mbox{for} \; \omega \in H^{1}_{0}, \;
\epsilon \in {\mathbb R}, \epsilon \not=0. $$
By direct calculation, this gives for the G\^{a}teaux derivative of $K$
at $\Psi$, evaluated at arbitrary $\omega$: 
$$
K^{\prime}(\Psi)[\omega] = \lim_{\epsilon \rightarrow 0} 
\frac{U^{\rho_{\epsilon}}\Psi_{0}(t) - 
U^{\rho}\Psi_{0}(t)}{\epsilon} = 
$$
\begin{equation}
\frac{2i}{\hbar}\int_{0}^{t} 
U^{\rho}(t,s)\left[\mbox{Re}({\bar \Psi} \omega) \ast W
+ ({\partial \Phi}{\partial \rho})  
\mbox{Re}({\bar \Psi} \omega) \right] 
U^{\rho}(s,0)\Psi_{0} \; ds.
\label{defKprime}
\end{equation}
Indeed, by direct calculation, we obtain
$$
\frac{U^{\rho_{\epsilon}}\Psi_{0}(t) - U^{\rho}\Psi_{0}(t)}{\epsilon} = 
$$
$$
\frac{2i}{\hbar}\int_{0}^{t} 
U^{\rho_{\epsilon}}(t,s)\left[\mbox{Re}({\bar \Psi} \omega) \ast W
+ ({\partial \Phi}{\partial \rho})  
\mbox{Re}({\bar \Psi} \omega) \right] 
U^{\rho}(s,0)\Psi_{0} \; ds \; +
$$
$$
\frac{i\epsilon}{\hbar}\int_{0}^{t} 
U^{\rho_{\epsilon}}(t,s)\left[|\omega|^{2} \ast W
+ (\partial \Phi/\partial \rho)  
|\omega|^{2} \right] 
U^{\rho}(s,0)\Psi_{0} \; ds. 
$$
An application of Lemma \ref{uniform} yields the limit. In fact, the first term
converges to the derivative, and the second term converges to zero; note
that the multiplier of $\epsilon$ remains bounded.

We now verify that
$K^{\prime}(\Psi)$ is a bounded linear operator on $C(J; H^{1}_{0})$.
By a direct estimate of the representation for 
$K^{\prime}(\Psi)[\omega]$, as given in (\ref{defKprime}), we have 
the norm estimate,
\begin{equation}
\label{adapest}
\|K^{\prime}(\Psi)\| \leq (2 C_{0} T/\hbar) 
 \|U^{\rho}\|_{\infty, H_{0}^{1}}^{2} 
\|\Psi_{0}\|_{H^{1}_{0}}. 
\end{equation}
The constant $C_{0}$ and the operator norm 
$ \|U^{\rho}\|_{\infty, H_{0}^{1}}^{2}$ are discussed in 
Lemma \ref{contractionconstant}.
This concludes the proof.
\end{proof}
\subsection{Lipschitz continuous Fr\'{e}chet differentiability}
\begin{proposition}
\label{LF}
The operator $K$ is continuously G\^{a}teaux differentiable  
and thus continuously Fr\'{e}chet differentiable. The derivative is, 
in fact, Lipschitz continuous on 
${\overline {B(0, r)}} \subset C(J; H^{1}_{0})$.
\end{proposition}
\begin{proof}
It is sufficient to prove the Lipschitz continuity of the G\^{a}teaux
derivative, as obtained in (\ref{defKprime}). 
For given $\Psi_{1}$ and $\Psi_{2}$, and 
arbitrary $\omega, \|\omega\|_{H^{1}_{0}} \leq 1$, 
we write the difference of
$K^{\prime}[\Psi_{1}](\omega)$ and 
$K^{\prime}[\Psi_{2}](\omega)$ as: 
$$
K^{\prime}[\Psi_{1}](\omega) -
K^{\prime}[\Psi_{2}](\omega) =
$$
$$
  \frac{2i}{\hbar}\int_{0}^{t} 
U^{\rho_{1}}(t,s)\left[\mbox{Re}({\bar \Psi_{1}} \omega) \ast W
+  \frac{\partial \Phi}{\partial \rho}  
\mbox{Re}({\bar \Psi_{1}} \omega)  \right] 
U^{\rho_{1}}(s,0)\Psi_{0} \; ds \;-
$$
$$
  \frac{2i}{\hbar}\int_{0}^{t} 
U^{\rho_{2}}(t,s)\left[\mbox{Re}({\bar \Psi_{2}} \omega) \ast W
+  \frac{\partial \Phi}{\partial \rho}  
\mbox{Re}({\bar \Psi_{2}} \omega)  \right] 
U^{\rho_{2}}(s,0)\Psi_{0} \; ds.
$$
The latter difference can be written as the sum of the three differences,
$D_{1}, D_{2}, D_{3}$, where $D_{1} =$
$$
  \frac{2i}{\hbar}\int_{0}^{t} 
U^{\rho_{1}}(t,s)\left[\mbox{Re}({\bar \Psi_{1}} \omega) \ast W
+  \frac{\partial \Phi}{\partial \rho}  
\mbox{Re}({\bar \Psi_{1}} \omega)  \right] 
(U^{\rho_{1}}(s,0)-U^{\rho_{2}}(s,0))\Psi_{0} \; ds, 
$$
where $D_{2} = $
$$
  \frac{2i}{\hbar}\int_{0}^{t} 
U^{\rho_{1}}(t,s)\left[\mbox{Re}(({\bar \Psi_{1}}-{\bar\Psi_{2}}) \omega)\ast W
+  \frac{\partial \Phi}{\partial \rho}  
\mbox{Re}(({\bar \Psi_{1}} -{\bar \Psi_{2}}) \omega)  \right] 
U^{\rho_{2}}(s,0)\Psi_{0} \; ds, 
$$
and where $D_{3} =$
$$
  \frac{2i}{\hbar}\int_{0}^{t} 
[U^{\rho_{1}}(t,s)-
U^{\rho_{2}}(t,s)]
\left[\mbox{Re}({\bar \Psi_{2}} \omega) \ast W
+  \frac{\partial \Phi}{\partial \rho}  
\mbox{Re}({\bar \Psi_{2}} \omega)  \right] 
U^{\rho_{2}}(s,0)\Psi_{0} \; ds.
$$
\begin{itemize}
\item
Estimation of $D_{1}$. 
\end{itemize}
We estimate from left to right inside the integral as follows. 
The uniform boundedness of the evolution operator, followed by the
combination of Lemma \ref{3.1} and Assumptions \ref{3.1def}, 
gives the estimate for the $D_{1}$ contribution:
$$
\frac{2rT\gamma_{T}}{\hbar} \beta \|U\|_{\infty, H^{1}_{0}} 
\|\Psi_{1} - \Psi_{2}\|_{C(J; H^{1}_{0})},
$$
where $\beta$ is defined below. Here, $r$ is defined in the
introduction in (\ref{definitionr})  
and $\gamma_{T}, \|U\|_{\infty, H^{1}_{0}}$ are discussed in Lemma
\ref{contractionconstant}.
\begin{itemize}
\item
Estimation of $D_{2}$. 
\end{itemize}
Again, we estimate from left to right inside the integral, and utilize 
the uniform boundedness of the evolution operator. 
The 
combination of Lemma \ref{3.1} and Assumptions \ref{3.1def} yields the
result. Specifically, we have 
the estimate for the $D_{2}$ contribution:
$$
\frac{2T}{\hbar} \beta \|U\|_{\infty, H^{1}_{0}}^{2}
\|\Psi_{0}\|_{H^{1}_{0}} 
\|\Psi_{1} - \Psi_{2}\|_{C(J; H^{1}_{0})}.
$$
\begin{itemize}
\item
Estimation of $D_{3}$. 
\end{itemize}
The reasoning is similar to that of the estimation of $D_{1}$.
We have 
the estimate for the $D_{3}$ contribution:
$$
\frac{2rT\gamma_{T}}{\hbar} \beta \|U\|_{\infty, H^{1}_{0}} 
\|\Psi_{1} - \Psi_{2}\|_{C(J; H^{1}_{0})}.
$$
It remains to define $\beta$.
If $\Phi$ is not included in the effective potential, then $\beta$ is
simply the constant $C_{0}$ appearing in (\ref{constantC}).
If $\Phi$ is included, then this value must be incremented by the constant
appearing in (\ref{conprod2}).
\end{proof}
\subsection{Invertibility}
The following proposition addresses the invertibility at the fixed point 
of 
$I - K^{\prime}(\Psi)$, on the space $C(J; H^{1}_{0})$. 
\begin{proposition}
\label{INV}
We denote by $\Psi$ the unique fixed point of $K$.
The operator,
$$
S^{\prime}(\Psi) = I - K^{\prime}(\Psi),
$$
is invertible on $C(J; H^{1}_{0})$. 
\end{proposition}
{\it Proof:}
We consider, in turn, the injective and surjective properties. By the open
mapping theorem, the inverse then exists and is a continuous linear
operator.
\begin{description}
\item[Injective Property]
\end{description}
We assume that there is  $\omega \in C(J; H^{1}_{0})$, 
such that
$$
S^{\prime}(\Psi)[\omega] = (I - K^{\prime}(\Psi))\omega = 0.
$$
We can apply Gronwall's inequality to the estimate,
\begin{equation}
\label{Gronpre}
\|\omega (\cdotp, t)\| \leq C \int_{0}^{t} \|\omega (\cdotp, s)\| \; ds,   
\end{equation}
where the norm is the $H^{1}_{0}$ norm, and $C$ is a fixed positive constant.
Gronwall's inequality yields that 
$\omega \equiv 0$.
The derivation of (\ref{Gronpre}) proceeds directly from 
Proposition \ref{Dif}.
\begin{description}
\item[Surjective Property]
\end{description}
We use a fixed point argument (even though the problem is linear).
Suppose $f$ is given in
$X = C(J; H^{1}_{0})$.
We consider the equation,
$$
S^{\prime}(\Psi)[\psi] = (I - K^{\prime}(\Psi))\psi = f,
$$
for $\psi \in X$. This is equivalent to a fixed point for
$$
L \psi = 
K^{\prime}(\Psi)\psi + f.
$$ 
$L$ is seen to be a strict contraction for $t$ sufficiently small by an
application of (\ref{adapest}) with $T \mapsto t$. 
By continuation in $t$, we obtain a fixed point $\psi$.
\subsection{Exact Newton iteration for the system}
We have obtained the following result.
\begin{theorem}
\label{NCT}
Under the hypotheses of Appendix A, 
the TDDFT model admits of Picard iteration
(successive approximation) 
for the mapping $K$, restricted to ${\overline {B(0, r)}}$,
 which is convergent to the solution 
$\Psi$ in $C(J; H^{1}_{0})$. 
If $\Phi$ is explicitly present in the potential $V_{\rm e}$, we also 
assume (the) Assumptions 
(\ref{3.1def}), and define $\kappa, \sigma$
by (\ref{definitionkappa}), (\ref{explicitsigma}), resp. Here, $S = I -
K$. If the starting iterate $u_{0}$ is chosen as stated in Corollary
\ref{exactcombined}, where $q = \gamma_{t} < 1$, then exact Newton iteration is
consistent and quadratically convergent  
in $C(J;H^{1}_{0})$. 
\end{theorem}
\begin{proof}
We have demonstrated the Lipschitz continuity of $K^{\prime}$ in
Proposition \ref{LF} and the invertibility of $I - K^{\prime}$ at the
solution $\Psi$ in Proposition \ref{INV}. 
Moreover, the Lipschitz constant 
can be estimated explicitly by the estimations of $D_{1}, D_{2}, D_{3}$ in
the proof of Proposition \ref{LF}. 
The result follows from Corollary
\ref{exactcombined}. 
\end{proof}
\begin{remark}
If 
$[I - K^{\prime}(\Psi)]^{-1}$ has a convergent Neumann series, then
$\kappa$ can be estimated explicitly. This case is discussed in the
following section. 
\end{remark}
\section{Approximate Newton Methods}
In this
section, we consider the reduction of the complexity 
of Newton iteration to
that of the approximation of the evolution operator itself. Given the
framework with which we have studied the TDDFT model, 
this is the ultimate reduction possible.
In order to do this, it will be necessary to interpret the mappings
$G_{z}$ of the Lemma of Appendix \ref{appendixC} 
as approximate right inverses of the
derivative of $S$, as used in that lemma. This, in turn, is accomplished
by the truncated Neumann series, which is becoming increasingly
relevant in the study of the Schr\"{o}dinger equation \cite{MF}. 
\subsection{Properties of an approximate Newton method}
\begin{definition}
\label{prec}
The operators $S$ and $\{G_{z}\}$ of the
Lemma of Appendix \ref{appendixC} 
are said to satisfy the conditions of an
approximate Newton method if, on the domain $B_{\delta}$, 
there is $M$ such that: 
\begin{enumerate}
\item
$S^{\prime}$ is Lipschitz continuous with constant $2M$.
\item
$G_{z}$ is uniformly bounded in norm by $M$.
\item
$G_{z}$ satisfies the following approximation of the identity condition: 
For each $z \in B_{\delta}$, the operator $G_{z}$ satisfies the operator 
inequality,
$$
\|I - S^{\prime}(z) G_{z} \| \leq M \|S(z)\|.
$$
It is understood that, if $S(z) = 0$, then $G_{z} = [S^{\prime}(z)]^{-1}$.
\end{enumerate}
\end{definition}
The following proposition is a consequence of \cite[Theorem 2.3]{J3}.
\begin{proposition}
\label{quadapp}
Suppose that the conditions of
Definition \ref{prec} are satisfied 
for a given closed ball $B_{\delta}$
with $0 < \alpha < 1$ selected. 
Given $h \leq 1/2$, 
choose $\sigma$, 
$$
\sigma = \frac{2(M + M^{3})}{h}\max\left(1, \; \frac{1}{(1 - \alpha)
\delta}\right),
$$ 
and choose $\kappa = M$.
If
$$
\|Su_{0}\| \leq \sigma^{-1},
$$
then the approximate Newton sequence satisfies
the inequalities (\ref{Kanone}, \ref{Kantwo}), and
is quadratically convergent, as specified in 
(\ref{R-quadratic}).
\end{proposition}
\subsection{An approximate Newton method based on residual estimation}
Definition \ref{prec} of the preceding section reveals the 
properties required of an approximate
Newton method in order to maintain classical quadratic convergence.
In this section, we introduce such a method associated with TDDFT.
It has the desired effect of
simplifying the formation of the algorithm in the case when the Neumann
series exists. We remark that we have established a  
general result in
Appendix C, which applies to a 
(closed) neighborhood $B_{\delta}$ of the fixed point. 

We require a preliminary lemma, which establishes a condition under which
the Neumann series exists. Mathematically, this involves the possible
restrictions of $H^{1}_{0}$-valued
functions to $[0, T^{\prime}], T^{\prime} \leq T$. 
For convenience, we will retain the notation, $B_{\delta}$, used in
Appendix C.
\begin{lemma}
\label{Neumannlocal}
There is a terminal
time $T = T^{\prime}$ such that 
\begin{equation*}
\|K^{\prime}(\Psi)\| < 1,
\end{equation*}
uniformly in $\Psi \in B_{\delta}$.
In fact, if the constant 
$C_{0}$, 
appearing in (\ref{adapest}), 
is chosen as discussed in 
Lemma \ref{contractionconstant}, 
and $\|U^{\rho}\|_{\infty, H^{1}_{0}}$ is bounded 
as in Lemma \ref{contractionconstant}, then $T^{\prime}$ may be chosen
to make the (resulting) rhs of  
(\ref{adapest}) less than one.
\end{lemma}
\begin{proof}
This is immediate from the estimate (\ref{adapest}) 
and the 
formulation of Lemma 
\ref{contractionconstant}. 
\end{proof}
\begin{definition}
\label{Neumannapp}
By the truncated Neumann series approximation to $(I -
K^{\prime}(\Psi))^{-1}$ is meant the expression,
\begin{equation}
\label{Neumannappeq}
G_{\Psi}\omega = [I + \sum_{k=1}^{n} (K^{\prime}(\Psi))^{k}](\omega),  
\end{equation} 
where $K^{\prime}(\Psi)[\omega]$ has been defined in (\ref{defKprime}).
The choice of $n$ may depend on $\Psi$. We will assume that
\begin{equation}
\label{Neumannsatisfied}
\|K^{\prime}(\Psi)\| \leq K_{0} < 1,
\end{equation}
uniformly in $\Psi$,
whenever this approximation is used.
\end{definition}
This will play the role of the approximate inverse. 
We consider  
the magnitude of the lhs of condition 3 of Definition \ref{prec}. 
\begin{lemma}
\label{AID}
The identity, 
$$
I - (I - K^{\prime}(\Psi))
(I + \sum_{k=1}^{n} (K^{\prime}(\Psi))^{k}) = 
(K^{\prime}(\Psi))^{n+1}, 
$$
holds. 
The term 
$
\|K^{\prime}(\Psi)\|
$
is estimated, as in  
(\ref{adapest}), 
by a positive constant times 
$T^{\prime}$.
\end{lemma}
\begin{proof}
The identity is routine. 
The estimation of
$
\|K^{\prime}(\Psi)\|
$
is given by the inequality (\ref{adapest}) with $T \mapsto T^{\prime}$. 
The conclusion follows.
\end{proof}
\begin{proposition}[Choice of constants in Definition \ref{prec}]
\label{adaptivepre}
For an appropriate terminal time, $T = T^{\prime}$,
chosen so that (\ref{Neumannsatisfied}) holds,
we define
$$
M = \max \left(
\frac{1}{1- K_{0}}, \frac{c}{2} \right).
$$
Here, $c$ is the Lipschitz constant of $K^{\prime}$, estimated in
Proposition \ref{LF}.
Then conditions one and two of Definition \ref{prec} hold.
The third condition holds if, for $\Psi$ not the fixed point of $K$,  
we choose $n$ to be the smallest positive integer such that
\begin{equation}
\label{choiceofN}
\|K^{\prime}(\Psi)\|_{C(J^{\prime}; H^{1}_{0})}^{n}  \leq 
\|(I - K)\Psi\|_{C(J^{\prime}; H^{1}_{0})}. 
\end{equation}
\end{proposition}
\begin{proof}
As previously remarked, $M$ is a bound for the norm of the approximate
inverse, and $2M$ is a bound for the Lipschitz constant of $S^{\prime}$.
It follows directly that conditions one and two hold. 
If we combine the identity in Lemma \ref{AID} with the definition of $n$
in (\ref{choiceofN}), we obtain
$$
\|I - (I - K^{\prime}(\Psi)) G_{\Psi}\| 
_{C(J^{\prime}; H^{1}_{0})} 
\leq K_{0} \|(I - K)\Psi\|
_{C(J^{\prime}; H^{1}_{0})}. 
$$ 
Since $K_{0} \leq M$, we conclude that condition 3 holds.
\end{proof}
The following theorem is a direct consequence of the discussion of this
section, 
specifically Propositions \ref{quadapp} and
\ref{adaptivepre}.
\begin{theorem}
Consider a time interval $J^{\prime} = [0, T^{\prime}]$, 
where $T^{\prime}$ has been
selected so that (\ref{Neumannsatisfied}) holds and, further, so that 
the contraction constant $\gamma_{T^{\prime}}$, associated with $K$, is less
than one on the closed ball $B_{r} = {\overline {B(0, r)}} \subset
C(J^{\prime}; H^{1}_{0})$. 
Suppose that a family $\{G_{\Psi}\}$ 
of approximate inverses is selected as in
Definition \ref{adaptivepre}, with $n = n(\Psi)$ chosen according to
(\ref{choiceofN}). Suppose that $M$ is defined as in Proposition 
\ref{adaptivepre}
and $\sigma$ as in 
Proposition \ref{quadapp}, where we identify $\delta$ with $r$.
If $u_{0}$ is selected by
Picard iteration (successive approximation) to satisfy 
$\|(I - K)u_{0}\|_{C(J^{\prime}; H^{1}_{0})} \leq \sigma^{-1}$, then
approximate Newton iteration is quadratically convergent. Specifically,
the sequence, for $k \geq 1$, and $n = n(k)$,
$$
u_{k} = u_{k-1} - 
(I + \sum_{j=1}^{n} (K^{\prime}(u_{k-1}))^{j}) (I - K)(u_{k-1}), 
$$
converges quadratically to the unique solution of the system.
\end{theorem}
\section{Conclusion}
The current article 
suggests successive approximation, proved in
\cite{J1}, combined with operator Newton iteration. 
The underlying operator and convergence  
theory for this analysis were demonstrated in \cite{J2,J3}.

Our results establish the classical Newton iteration theory, with
quadratic convergence, under hypotheses which naturally extend those 
used to prove existence and uniqueness. Moreover, our 
theory also permits an approximate inverse for the derivative mapping, and
this is implemented via the truncated Neumann series. 
The norm condition, for the use of the truncated Neumann series as
approximate inverse, appears severe,
since it represents an implicit restriction on the
length of the time interval. However, a similar restriction  
appears in the existence/uniqueness theory of \cite{J1}; 
continuation in a finite number of time steps is possible to obtain a
global solution. It is thus natural to employ the exact or approximate Newton
method in combination with Picard iteration (successive approximation). 

Although this intermediate study is quite far from the algorithms of 
scientific computation,
it is compatible with numerical methods as illustrated in
\cite{CP} and implemented in \cite{JP}. It is also compatible with the 
numerical methods employed in the scientific literature cited here. 
Newton methods have been applied to quantum models previously; see \cite{CBDW} 
for a control theory application to 
the quantum Liouville-von Neumann master equation. For TDDFT, however, 
we believe that we have laid the foundation for future refinements defined 
by the methods of approximation theory. 

\appendix
\section{Hypotheses for the Hamiltonian}

We make the following assumptions for the existence/uniqueness theory
described in Theorem \ref{EU}.
\begin{itemize}
\item
$\Phi$ is assumed nonnegative
for each $t$, and continuous in $t$ on 
$H^{1}$ and bounded in $t$ into $W^{1,3}$. 
The continuity on $H^{1}$ is consistent with the zero-force law as
defined in \cite[Eq.\thinspace (6.9)]{U}:
$$
\int_{\Omega} \rho({\bf x},t) \nabla \Phi({\bf x}, t, \rho) = 0.
$$
\item
The derivative $\partial \Phi/\partial t = 
\phi$ is assumed measurable, and bounded in its arguments. 
\item
Furthermore, the following smoothing condition
is assumed, 
expressed in a (uniform) Lipschitz norm condition:
$\forall t \in [0,T]$,
\begin{equation}
\|[\Phi(t, |\Psi_{1}|^{2}) - \Phi(t, |\Psi_{2}|^{2})]\psi\|_{H^{1}}
\leq 
C \|\Psi_{1} - \Psi_{2}\|_{H^{1}} \|\psi\|_{H^{1}_{0}}. 
\label{ecfollowsH}
\end{equation}
Here, $\psi$ is arbitrary in $H^{1}_{0}$.
\item
The so-called external potential $V$ is assumed to be nonnegative and 
continuously
differentiable on the closure of the space-time domain. 
\end{itemize}
We remark that the hypotheses of nonnegativity for $V$ and $\Phi$ are for
convenience only.
\section{Notation and Norms}
We employ complex Hilbert spaces 
in this article. 
$$
L^{2}(\Omega) = \{f = (f_{1}, \dots, f_{N})^{T}: |f_{j}|^{2} \;
\mbox{is integrable on} \; \Omega \}.
$$
Here, $f: \Omega \mapsto {\mathbb C}^{N}$ and 
$
|f|^{2} = \sum_{j=1}^{N} f_{j} {\overline {f_{j}}}.
$
The inner product on $L^{2}$ is
$$
(f,g)_{L^{2}}=\sum_{j=1}^{N}\int_{\Omega}f_{j}(x){\overline {g_{j}(x)}} \; dx.
$$
However, $\int_{\Omega} fg$ is interpreted as 
$$
\sum_{j=1}^{N} \int_{\Omega} f_{j} g_{j} \;dx.
$$
For $f \in L^{2}$, as just defined, if each component $f_{j}$ satisfies
$
f_{j} \in H^{1}_{0}(\Omega; {\mathbb C}), 
$
we write $f \in H^{1}_{0}(\Omega; {\mathbb C}^{N})$, or simply, 
$f \in H^{1}_{0}(\Omega)$.
The inner product in $H^{1}_{0}$ is 
$$
(f,g)_{H^{1}_{0}}=
(f,g)_{L^{2}}+\sum_{j=1}^{N}\int_{\Omega} 
\nabla f_{j}(x) \cdotp {\overline {\nabla g_{j}(x)}} \; dx.
$$
$\int_{\Omega} \nabla f \cdotp \nabla g$ is interpreted as
$$
\sum_{j=1}^{N}\int_{\Omega} 
\nabla f_{j}(x) \cdotp \nabla g_{j}(x) \; dx.
$$
Finally, $H^{-1}$ is defined as the dual of $H^{1}_{0}$, and its
properties are discussed at length in \cite{Adams}. 
The Banach space $C(J; H^{1}_{0})$ is defined in the traditional manner:
$$
C(J; H^{1}_{0}) = \{u:J \mapsto H^{1}_{0}: u(\cdotp) \mbox{is
continuous}\}, \;
\|u\|_{C(J; H^{1}_{0}} = \sup_{t \in J} \|u(t)\|_{H^{1}_{0}}.
$$
\section{A General Quadratic Convergence Result in Banach Space}
\label{appendixC}
We cite a weaker version, with different notation,
of a result proved in \cite[Lemma 2.2]{J3}.
Additional references to the earlier literature can be found there.
{\bf Lemma}. {\it
Suppose that $B_{\delta}:= \overline {B(x_{0}, \delta)}$ 
 is a closed ball in a Banach space $X$ and
$S$ is a continuous mapping from $B_{\delta}$ to a Banach space $Z$. 
Suppose that 
$u_{0} \in B_{\alpha \delta}$ where $0 < \alpha < 1$, and 
$\|S(u_{0})\| \leq \sigma^{-1}$. Suppose that a family 
$\{G_{z}\}$ of bounded linear  
operators is given, for each $z$ in the range of $S$, 
where $G_{z}: Z \mapsto X$. 
Define the iterates,
$$
u_{k} - u_{k-1} = -G_{u_{k-1}}S(u_{k-1}), \; k = 1,2, \dots
$$ 
Let $h$ be chosen so that $h \leq 1/2$, and set
$$
t^{\ast} =(h \sigma)^{-1} (1 - \sqrt{1 - 2h}). 
$$
The procedure is consistent, i.\thinspace e.\thinspace, $\{u_{k}\} \subset
B_{\delta}$, if the inequalities,
\begin{eqnarray}
\|u_{k}-u_{k-1}\|& \leq & \kappa
\|{S}(u_{k-1})\|, \;\; k \geq 1, \label{Kanone} \\
\|{S}(u_{k})\| & \leq & 
\frac{h\sigma}{2}\|{S}(u_{k-1})\|^{2},
\;\; k \geq 1, \label{Kantwo}
\end{eqnarray}
hold for some $\kappa \leq (1 - \alpha)\delta/t^{\ast}$.  
Moreover, the sequence converges to a root $u$ with 
the error estimate,
\begin{equation}
\|u - u_{k}\| \leq \frac{\kappa}{h \sigma}
\frac{(1-\sqrt{1-2h})^{2^{k}}}{2^{k}}.
\label{R-quadratic}
\end{equation}
}


\begin{thebibliography}{100}
\bibitem{RG}  E.~Runge and E.K.U.~Gross,
Density functional theory for time dependent systems. 
Phys. Rev. Lett. {\bf 52} (1984), 997--1000.

\bibitem{U} C.A.~Ullrich, 
{\em Time-Dependent Density-Functional Theory: Concepts and Applications}.
Oxford University Press, 2012.

\bibitem{KS} W.~Kohn and L.J.~Sham,
Self-consistent equations including exchange and correlation effects.
Phys. Rev. {\bf 140} (1965), A1133--A1138.

\bibitem{KV} W.~Kohn and P.~Vashista, 
General density functional theory. In, 
\emph{Theory of the Inhomogeneous
Electron Gas}. Ch. 2, Plenum Press, New York, 1983. 

\bibitem{JP} 
J.W.~Jerome and E.~Polizzi, Discretization of time dependent quantum 
systems: Real time
propagation of the evolution operators. Appl. Anal. {\bf 93} (2014), 
2574--2597.

\bibitem{J1} J.W.~Jerome, Time-dependent closed quantum systems: Nonlinear
Kohn-Sham potential operators and weak solutions, J. Math. Anal. Appl. {\bf 429}
(2015), 995--1006.  

\bibitem{CMR}
A.~Castro, M.A.L.~Marques, and A.~Rubio, Propagators for the
time-dependent Kohn-Sham equations. J. Chem. Phys. {\bf 121} (2004),
3425--3433.
 
\bibitem{tddft} A.~Castro and M.A.L.~Marques, 
 Time dependent density functional theory, In, Lec.
Notes in Phys. {\bf 706} (2006), 197--210.

\bibitem{YB} K.~Yabana,  T.~Nakatsukasa,  J.-I.~Iwata, and G.F.~Bertsch,
Real-time, real-space implementation of the 
linear response time-dependent density-functional theory.
Phys. Stat. Sol. (b) {\bf 243} No. 5 (2006), 1121--1138. 

\bibitem{LHP}
L.~Lehtovaara, V.~Havu, and M.~Puska, All-electron time-dependent
density functional theory with finite elements: Time-propagation approach.
J. Chem. Phys. {\bf 135} (2011), 154104.

\bibitem{MG}
M.A.L.~Marques and E.K.U.~Gross,
Time dependent density functional theory.
Annu. Rev. Phys. Chem. (2004), 427--455. 

\bibitem{CH}
T.~Cazenave and A.~Haraux, {\it An Introduction to Semilinear Evolution
Equations}. Revised Ed., Oxford Science Publications, 1998. 

\bibitem{Caz} T.~Cazenave, {\em Semilinear Schr\"{o}dinger Equations}. 
Courant Institute Lec. Notes 10, 2003, Published by the American
Mathematical Society.

\bibitem{K1} 
T.~Kato, Linear equations of hyperbolic type. J. Fac. Sc. Univ. Tokyo {\bf
17} (1970), 241--258.

\bibitem{K2}  
T.~Kato, Linear equations of hyperbolic type II. J. Math. Soc. Japan 
{\bf 25} (1973), 648--666.

\bibitem{J2}  J.W.~Jerome, {\em Approximation of Nonlinear Evolution
Equations}. Academic Press, 1983. 

\bibitem{J3} J.W.~Jerome, Approximate Newton methods and homotopy for stationary
operator equations. Constr. Approx. {\bf 1} (1985), 271--285. 

\bibitem{G} S.~Goldberg, Ranges and inverses of perturbed linear operators. 
Pacific J. Math. {\bf 9} (1959),
701--706.

\bibitem{H} 
R.S.~Hamilton, The inverse function theorem of Nash and Moser. Bull. Amer.
Math. Soc. (NS) {\bf 7} (1982), 65--222.

\bibitem{Rudin}
W.~Rudin, {\em Functional Analysis}, McGraw-Hill, 1973.

\bibitem{GT}
D.~Gilbarg and N.S.~Trudinger, {\em Elliptic Partial Differential
Equations of Second Order}. Grundlagen der mathematischen Wissenschaften
224, Springer-Verlag, 1977.

\bibitem{MF}
F.~Mera and S.~Fulling, Convergence of the Neumann series for the
Schr\"{o}dinger equation and general Volterra equations in Banach spaces.
In, {\em Advances in Quantum Mechanics}, Ch. 11, P.~Bracken, ed., InTech,
Rijeka, Croatia, 2013, 247--268.

\bibitem{CP} Z.~Chen and E.~Polizzi, Spectral-based propagation schemes  
for time-dependent quantum systems with applications to carbon nanotubes.
Phys. Rev. B {\bf 82} (2010), 205410 (8 pages).

\bibitem{CBDW}
G.~Ciaramella, A.~Borzi, G.~Dirr, and D.~Wachsmuth, Newton methods for the
optimal control of closed quantum spin systems. SIAM J. Sci. Comput. {\bf
37} (2015), no. 1, A319--A346.

\bibitem{Adams}
R.~Adams and J.~Fournier, {\it Sobolev Spaces}.
2nd. ed., Elsevier/Academic Press, 2003.

\end{thebibliography}
\end{document}